\documentclass[11pt]{amsart}

\usepackage{amsmath,amssymb,amsthm}
\usepackage[margin=1in]{geometry}
\usepackage{hyperref}

\newtheorem{theorem}{Theorem}[section]
\newtheorem{lemma}[theorem]{Lemma}
\newtheorem{corollary}[theorem]{Corollary}
\newtheorem{proposition}[theorem]{Proposition}

\newtheorem{definition}[theorem]{Definition}
\newtheorem{remark}[theorem]{Remark}

\DeclareMathOperator{\supp}{supp}
\DeclareMathOperator{\union}{\cup}

\title[The  Pego Theorem for the Hilbert--Schmidt Class]{The  Pego Theorem for the Hilbert--Schmidt Class}
\author{Yaogan Mensah}
\address{Department of Mathematics, University of Lom\'e, Lom\'e, Togo}
\email{mensahyaogan2@gmail.com, ymensah@univ-lome.tg}
\date{}

\begin{document}

\begin{abstract}
This paper establishes an  operator-theoretic version of Pego's compactness theorem within the framework of quantum harmonic analysis on general locally compact abelian phase spaces. We show that a bounded set of Hilbert-Schmidt operators is precompact if and only if it is uniformly equicontinuous under phase-space shifts and its Fourier-Weyl transform is uniformly equicontinuous on the dual phase space. We provide applications to quantum physics.  
\end{abstract}

\maketitle

\section{Introduction}
\label{sec:intro}
In 1985, Pego proved a compactness criterion for $L^2(\mathbb R^n)$: a bounded set
$F\subset L^2(\mathbb R^n)$ is precompact if and only if both $F$ and its Fourier
transform $\widehat{F}$ satisfy a joint decay and equicontinuity condition \cite{Pego1985}. This is a
refinement of the Fr\'echet-Riesz-Kolmogorov theorem (extended by Weil to locally compact groups \cite{Weil1940}). 

G\'orka extended Pego's theorem to arbitrary locally compact abelian (LCA) groups $G$, replacing the ordinary
Fourier transform by the group Fourier transform \cite{Gorka2014, Gorka2016}. Kumar later proved a version for
 compact groups, where the dual object $\widehat{G}$ is the
discrete set of irreducible unitary representations \cite{Kumar2024}. Lakmon and Mensah extended the results of Kumar to Hilbert space-valued functions \cite{Lakmon2026}. Other works on Pego type theorem include \cite{Dorfler2002, Horvath2022, Krukowski2020}.

On a completely different axis, Werner introduced quantum harmonic analysis (QHA) on the phase space $\Xi=\mathbb R^{2n}$: convolutions of functions and operators, a
Fourier-Weyl  transform taking trace-class operators to functions on
phase space, and versions of the Riemann-Lebesgue lemma, the Hausdorff-Young
inequality, Plancherel's theorem, and Wiener's approximation theorem for operators
\cite{Werner1984}. Fulsche and Galke extended this entire machinery to phase
spaces $\Xi$ that are arbitrary LCA groups equipped with a Heisenberg multiplier,
recovering Werner's results 
\cite{FulscheGalke2025}.

The main objective of this paper is to establish a non-commutative, operator-theoretic version of Pego's theorem within the expansive framework of quantum harmonic analysis. Specifically, we establish a criterion for precompactness in the Hilbert-Schmidt class $\mathcal{T}^2(\mathcal{H})$ using the interplay between structural properties of operator families and their  Fourier-Weyl transforms. Our central result  demonstrates that a bounded subset $F \subset \mathcal{T}^2(\mathcal{H})$ is precompact if and only if $F$ is uniformly $\mathcal{T}^2$-equicontinuous under unitary phase-space shifts, and its Fourier-Weyl transform $\mathcal{F}_U(F)$ is uniformly $L^2(\widehat{\Xi})$-equicontinuous. 

To demonstrate the  utility and physical relevance of our results, we apply our criterion to two distinct situations. First, we establish the norm-precompactness of multi-mode thermal-state families under finite physical energy constraints in quantum information. Second, we verify the precompactness, hence tightness, of tomographic quantum state estimators in quantum statistics.

The paper is organized as follows. Section \ref{sec:classical} collects the classical background on precompactness in $L^p$-spaces.  Section \ref{sec:prelim} gathers the mathematical preliminaries of quantum harmonic analysis on a general Heisenberg phase space.  Section \ref{sec:quantumPego} introduces uniform equicontinuity and uniform decay for operator families and culminates in the quantum type of  Pego theorem. Finally, Section \ref{sec:applications} applies the results to two settings: the precompactness of energy bounded multi-mode thermal state families in quantum information  and the tightness of tomographic quantum state estimators in quantum statistics.

\section{Precompactness in $L^p$-spaces}
\label{sec:classical}
To keep the paper reasonably self-contained, we gather here the classical results on
precompactness in $L^p$-spaces:  the Fr\'echet-Kolmogorov-Riesz theorem on
$\mathbb{R}^n$, its extension by Weil to an arbitrary locally compact group, and the
unmixed characterizations due to Pego and G\'orka. We begin with the following definition.

\begin{definition}
\label{def:relatively-compact}
Let $(X, d)$ be a metric space. A subset $A \subset X$ is said to be \emph{relatively compact} (or precompact) in $X$ if its closure $\overline{A}$ is a compact subset of $X$.
\end{definition}

\begin{theorem}[Fr\'echet-Kolmogorov-Riesz]
\label{thm:kolmogorov-riesz}
Let $1 \le p < \infty$ and let $F \subset L^p(\mathbb{R}^n)$. Then $F$ is relatively compact in $L^p(\mathbb{R}^n)$ if and only if the following three conditions hold:
\begin{enumerate}
    \item[(i)]  There exists $M > 0$ such that
       $$\|f\|_{L^p} \le M \quad \text{for all } f \in F.
    $$
    \item[(ii)]  For every $\varepsilon > 0$, there exists $R > 0$ such that
    $$
        \int_{|x| > R} |f(x)|^p \, dx < \varepsilon^p \quad \text{for all } f \in F.
    $$
    \item[(iii)] For every $\varepsilon > 0$, there exists $\delta > 0$ such that
    $$
        \|\tau_y f - f\|_{L^p} < \varepsilon \qquad \text{for all } y \in \mathbb{R}^n \text{ with } |y| < \delta, \text{ and all } f \in F,
    $$
    where $\tau_y f(x) = f(x - y)$.
\end{enumerate}
\end{theorem}

This result is due, in various forms, to Fr\'echet, Kolmogorov, and Riesz between
1928 and 1933 (see \cite{HancheOlsenHolden2010} for the history). It was extended by
Weil in 1940 \cite{Weil1940} from $L^p(\mathbb{R}^n)$ to $L^p(G)$ for an arbitrary
locally compact group $G$.

\begin{theorem}[Weil \cite{Weil1940}]
\label{thm:weil-compactness}
Let $G$ be a locally compact Hausdorff group equipped with a left Haar measure, and let $1 \le p < \infty$. A subset $F \subset L^p(G)$ is relatively compact if and only if the following three conditions hold:
\begin{enumerate}
    \item[(i)]  There exists $M > 0$ such that
    $$
 \|f\|_{L^p(G)} \le M \quad \text{for all } f \in F.
    $$
    \item[(ii)]  For every $\varepsilon > 0$, there exists a compact set $K \subset G$ such that
    $$
        \|f - f {1}_K\|_{L^p(G)} < \varepsilon \quad \text{for all } f \in F.
    $$
    \item[(iii)]  For every $\varepsilon > 0$, there exists an open neighbourhood $V$ of the identity $e \in G$ such that
    $$
        \|L_x f - f\|_{L^p(G)} < \varepsilon \qquad \text{for all } x \in V \text{ and all } f \in F,
    $$
    where $L_x f(y) = f(x^{-1} y)$.
\end{enumerate}
\end{theorem}

Pego observed that for $p=2$ and $G=\mathbb R^n$, the Plancherel
theorem yields  one of the two conditions 
imposed  on $F$ and its Fourier transform $\widehat{F}=\{\widehat{f}:
f\in F\}$.

\begin{theorem}[Pego \cite{Pego1985}]
\label{thm:pego-compactness}
Let $F \subset L^2(\mathbb{R}^n)$ be bounded. The following assertions
are equivalent.
\begin{enumerate}
\item[(i)] $ F$ is relatively compact in $L^2(\mathbb{R}^n)$.
\item[(ii)] \text{Decay of $F$ and of $\widehat{F}$:} for
every $\varepsilon>0$ there are $R,\rho>0$ such that
$$
 \int_{|x| > R} |f(x)|^2 \, dx < \varepsilon^2
        \quad\text{and}\quad
        \int_{|\xi| > \rho} |\widehat{f}(\xi)|^2 \, d\xi < \varepsilon^2
        \quad \text{for all } f \in F.
$$
\item[(iii)] \text{Equicontinuity of $ F$ and of $\widehat{
F}$:} for every $\varepsilon>0$ there is $\delta>0$ such that
$$
\|\tau_yf-f\|_{L^2}<\varepsilon \quad\text{and}\quad
\|\tau_\eta\widehat f-\widehat f\|_{L^2}<\varepsilon
\qquad\text{for all } f\in F,\ |y|,|\eta|<\delta.
$$
\end{enumerate}
\end{theorem}

G\'orka \cite{Gorka2014} extended Theorem \ref{thm:pego-compactness} from
$\mathbb{R}^n$ to an arbitrary LCA group $G$. Let us mention that a superfluous technical hypothesis in the original argument was
removed by G\'orka and Kostrzewa two years later \cite{Gorka2016}.

\begin{theorem}[G\'orka \cite{Gorka2014}; G\'orka--Kostrzewa \cite{Gorka2016}]
\label{thm:gorka-pego}
Let $G$ be a locally compact abelian group, and let
$\widehat{G}$ denote its Pontryagin dual, with Fourier transform $f \mapsto
\widehat{f}$ mapping $L^2(G) \to L^2(\widehat{G})$. Let $F \subset
L^2(G)$ be bounded. The following assertions are equivalent.
\begin{enumerate}
\item[(i)] $ F$ is relatively compact in $L^2(G)$.
\item[(ii)] \text{Decay of $ F$ and of $\widehat{ F}$:} for
every $\varepsilon > 0$, there exist compact sets $K \subset G$ and $\widehat K
\subset \widehat G$ such that
$$
        \int_{G \setminus K} |f(x)|^2 \, dx < \varepsilon^2
        \quad\text{and}\quad
        \int_{\widehat{G} \setminus \widehat{K}} |\widehat{f}(\chi)|^2 \, d\chi < \varepsilon^2
        \qquad \text{for all } f \in F.
$$
\item[(iii)] \text{Equicontinuity of $ F$ and of $\widehat{
F}$:} for every $\varepsilon>0$ there is an open neighbourhood $V$ of $e\in G$
and an open neighbourhood $\widehat{V}$ of the trivial character in $\widehat{G}$
such that
$$
\|\tau_xf-f\|_{L^2(G)}<\varepsilon \ \text{ for all } x\in V,
\quad\text{and}\quad
\|\tau_\chi\widehat f-\widehat f\|_{L^2(\widehat G)}<\varepsilon \ \text{ for all
} \chi\in\widehat V,
$$
for all $f\in F$.
\end{enumerate}
\end{theorem}

The aim of this paper is to obtain the operator analogue of Theorems
\ref{thm:pego-compactness}--\ref{thm:gorka-pego}.

\section{Preliminaries on quantum Harmonic analysis }\label{sec:prelim}

We recall the setup of quantum harmonic
analysis on a general Heisenberg phase space due to Fulsche and Galke
\cite{FulscheGalke2025}. We refer to that paper for all unattributed claims in this section.

\subsection{Schatten classes}
For a complex Hilbert space $\mathcal{H}$, let $\mathcal{L}(\mathcal{H})$ denote the
bounded linear operators on $\mathcal{H}$ and $K(\mathcal{H})$ the compact ones. For
$1\le p<\infty$, the {Schatten $p$-class} is
\begin{equation*}
\mathcal{T}^p(\mathcal{H}) = \Big\{A\in K(\mathcal{H}) : \|A\|_{\mathcal{T}^p}=
\Big(\sum_{n=1}^\infty s_n(A)^p\Big)^{1/p}<\infty\Big\},
\end{equation*}
where $(s_n(A))_n$ are the singular values of $A$ (ordered decreasingly and counting multiplicities). We set $\mathcal{T}^\infty(\mathcal{H})=\mathcal{L}(\mathcal{H})$ with $\|\cdot\|_{\mathcal{T}^\infty}=\|\cdot\|_{\mathrm{op}}$, the operator norm. Each $(\mathcal{T}^p(\mathcal{H}),\|\cdot\|_{\mathcal{T}^p})$,
$1\le p\le\infty$, is a complex Banach space. $\mathcal{T}^1(\mathcal{H})$ is called the trace class and 
$\mathcal{T}^2(\mathcal{H})$  the Hilbert-Schmidt class. 

\subsection{Phase space and representation}

Let $\Xi$ be an LCA group with neutral element denoted $0$ and $\widehat{\Xi}$ its Pontryaging dual with neutral element denoted 1. Let  $m:\Xi\times\Xi\to\mathbb T$ be a
separately continuous Heisenberg multiplier; that is,  the bicharacter
$\sigma(\cdot,\cdot)$ defined by $\sigma(x,y)=\displaystyle\frac{ m(x,y)}{m(y,x)}$ induces a topological isomorphism
$\Xi\ni x\mapsto\sigma(x,\cdot)\in\widehat{\Xi}$, so $\Xi$ is canonically self-dual. By
the Mackey-Stone-von Neumann theorem \cite{Mackey1949},  there is a unique (up to unitary equivalence)
irreducible projective unitary representation $(U,\mathcal{H})$ of $\Xi$ with
multiplier $m$; that is
\begin{equation*}
U_xU_y = m(x,y)\,U_{x+y}, \quad x,y\in\Xi. 
\end{equation*}
Consequently, 
\begin{equation*}
U_xU_y=\sigma(x,y)U_yU_x, \quad x,y\in\Xi.
\end{equation*}
We will always assume that $m(x, 0) = m(0, x) = 1$ for every $x\in \Xi$, or equivalently $U_0 = I_{\mathcal{H}}$, the identity operator.

We assume throughout that the irreducible projective unitary representation $(U,\mathcal{H})$ is:
\begin{itemize}
\item {square-integrable}, so that Haar measure on $\Xi$ can be normalized, 
making the orthogonality relations hold on $\mathcal{H}$ \cite[Theorem 2.2]{FulscheGalke2025};
\item {integrable}: there is $0\ne\varphi\in\mathcal{H}$ with
$x\mapsto\langle U_x\varphi,\varphi\rangle\in L^1(\Xi)$;
\item {strongly continuous}: $x\mapsto U_x\varphi$ is continuous
$\Xi\to\mathcal{H}$ for every $\varphi\in \mathcal{H}$.
\end{itemize}
These hold, for instance, whenever $\Xi=G\times\widehat G$ for an LCA group $G$, with $U$ the
Schr\"odinger representation on $\mathcal{H}=L^2(G)$.

\subsection{Convolution of a function and an operator}
For $A\in\mathcal L(\mathcal{H})$ and $x\in\Xi$, define the shift $\alpha_x(A)$ of $A$ by  
$$\alpha_x(A)=U_xAU_x^*.$$

For $f\in L^1(\Xi)$ and $A\in\mathcal{T}^1(\mathcal{H})$, the convolution $f*A$ given by the Bochner integral
\begin{equation*}
f*A= \int_\Xi f(x)\,\alpha_x(A)\,dx
\end{equation*}
defines an element of $\mathcal{T}^1(\mathcal{H})$, with 
\begin{equation*}
\|f*A\|_{\mathcal{T}^1}\le\|f\|_{L^1(\Xi)}\|A\|_{\mathcal{T}^1}.
\end{equation*}
The same integral defines an element of $\mathcal L(\mathcal{H})$ for
$A\in\mathcal L(\mathcal{H})$, with $\|f*A\|_{\mathrm{op}}\le\|f\|_{L^1(\Xi)}\|A\|_{\mathrm{op}}$. By
interpolation  
\begin{equation}
\|f*A\|_{\mathcal{T}^p} \le \|f\|_{L^1(\Xi)}\,\|A\|_{\mathcal{T}^p}, \quad 1\le
p\le\infty. 
\end{equation}
Therefore,  $f*A\in\mathcal{T}^p(\mathcal{H})$ for every $A\in\mathcal{T}^p(\mathcal{H})$.
\subsection{The two Fourier transforms}

For $f\in L^1(\Xi)$, the symplectic Fourier transform is
\begin{equation*}
\mathcal{F}_\sigma(f)(\xi) = \int_\Xi \sigma(x,\xi)f(x)dx, \quad \xi\in\widehat{\Xi}.
\end{equation*}
It satisfies the classical Riemann-Lebesgue lemma; that is, $\mathcal{F}_\sigma(f)\in
C_0(\widehat{\Xi})$ with $$\|\mathcal{F}_\sigma(f)\|_\infty\le\|f\|_{L^1(\Xi)}.$$ (Here $C_0(\widehat{\Xi})$ is the space of  complex continuous functions on $\widehat{\Xi}$ which vanish at $\infty$).

For $A\in\mathcal{T}^1(\mathcal{H})$, the Fourier-Weyl transform is
\begin{equation*}
\mathcal{F}_U(A)(\xi) = \text{tr}(AU_\xi^*), \quad \xi\in\widehat{\Xi}
\end{equation*}
(here, following \cite{FulscheGalke2025}, $\widehat{\Xi}$ denotes $\Xi$ topologically but with
the Haar measure renormalized so that Plancherel's theorem below is unitary). 

For a
rank-one operator $A=\varphi\otimes\psi$, $$\mathcal{F}_U(A)(\xi)=\langle\varphi,U_\xi
\psi\rangle$$ where $\langle \cdot,\cdot \rangle$ is the scalar product in $\mathcal{H}$.

\subsection{Facts about $\mathcal{F}_U$}
\begin{itemize}
\item [\textbf{(Mod)}](\emph{Modulation}). For $A\in\mathcal{T}^1(\mathcal{H})$,  $$\mathcal{F}_U(\alpha_x(A))(\xi)=\sigma(x,\xi)\mathcal{F}_U(A)(\xi), \quad x\in \Xi, \,\xi\in \widehat{\Xi}.$$
\item[\textbf{(Conv)}](\emph{Convolution}). For $f\in L^1(\Xi)$, $A\in\mathcal{T}^1(\mathcal{H})$,
$$\mathcal{F}_U(f*A) = \mathcal{F}_\sigma(f)\cdot\mathcal{F}_U(A).$$

\item[\textbf{(HY)}] \emph{(Hausdorff-Young for $\mathcal{F}_U$).} For $1\le  p\le 2$ with conjugate exponent $q$,
\begin{equation*}
\|\mathcal{F}_U(A)\|_{L^q(\widehat{\Xi})} \le \|A\|_{\mathcal{T}^p}, \quad
A\in\mathcal{T}^p(\mathcal{H}).
\end{equation*}

\item[\textbf{(PL)}] \emph{(Plancherel theorem for operators).} 
$\mathcal{F}_U$ extends to a unitary operator
\begin{equation*}
\mathcal{F}_U : \mathcal{T}^2(\mathcal{H}) \longrightarrow L^2(\widehat{\Xi}),
\end{equation*}
with unitary inverse $$\mathcal{F}_U^{-1}(g)=\int_{\widehat{\Xi}}g(\xi)U_\xi d\xi,$$ which
itself satisfies analogues of the Riemann-Lebesgue and Hausdorff-Young estimates:
\item[ \textbf{(RL$^{-1}$)}]\emph{(Inverse Riemann-Lebesgue).}
For $g\in L^1(\widehat{\Xi})$, $\mathcal{F}_U^{-1}(g)\in K(\mathcal{H})$ with
$$\|\mathcal{F}_U^{-1}(g)\|_{\mathrm{op}}\le\|g\|_{L^1(\widehat{\Xi})};$$ 

\item[\textbf{(HY$^{-1}$)}]\emph{(Inverse Hausdorff-Young).} For
$1\le p\le 2$ with conjugate exponent $q$,
$$\|\mathcal{F}_U^{-1}(g)\|_{\mathcal{T}^q}\le\|g\|_{L^p(\widehat{\Xi})}, \quad g\in
L^p(\widehat{\Xi}).$$ 
\end{itemize}

\section{The quantum Pego theorem}\label{sec:quantumPego}
We recall the following definition. 

\begin{definition}[Totally bounded set]
\label{def:totally-bounded}
Let $(X, d)$ be a metric space. A subset $A \subset X$ is said to be \emph{totally bounded}  if, for every $\varepsilon > 0$, there exists a finite set of points $x_1, \dots, x_n \in X$ such that
\[
    A \subset \bigcup_{i=1}^{n} B(x_i, \varepsilon),
\]
where $B(x_i, \varepsilon) = \{ y \in X : d(x_i, y) < \varepsilon \}$. 
\end{definition}

\begin{remark}
\label{rem:totally-bounded}{\rm
 For a subset of a complete metric space, precompactness is equivalent to
total boundedness. We use this
equivalence repeatedly throughout.  
}\end{remark}

We now formulate, for operator families, the notion of uniform equicontinuity and uniform decay  that drive the Pego-type theorem for operators.

\begin{definition}
\label{def:opeq}
Let $1\le p\le 2$ and  $F\subset\mathcal{T}^p(\mathcal{H})$. We say that $F$ is
uniformly $\mathcal{T}^p$-equicontinuous if for every $\varepsilon>0$, there is
an open neighbourhood $V$ of $0\in\Xi$ such that
\begin{equation*}
\|\alpha_x(A)-A\|_{\mathcal{T}^p} < \varepsilon, \quad \text{for all } A\in F,\ x\in V.
\end{equation*}
\end{definition}

\begin{definition}
\label{def:funceq}

\begin{itemize}
\item[(i)] Let $S\subset L^q(\widehat{\Xi})$, $2\le q\le \infty$.  We say that $S$ is uniformly $L^q(\widehat{\Xi})$-equicontinuous if for every
$\varepsilon>0$ there is an open neighbourhood $V$ of the trivial character $1\in \widehat{\Xi}$ such that
$$\|L_\eta g-g\|_{L^q}<\varepsilon, \text{ for all } g\in S,\ \eta\in V.$$

\item[(ii)] We say  that $S$ has
uniform $L^q(\widehat{\Xi})$-decay if for every $\varepsilon>0$ there is a compact
$K\subset\widehat{\Xi}$ such that $$\int_{\widehat{\Xi}\setminus K}|g(\xi)|^q\,d\xi < \varepsilon,  \text{ for all } g\in S.$$
\end{itemize}
\end{definition}

\bigskip

We outline some facts to be used in the proof of the subsequent results.

\begin{itemize}
\item[(F1)] 
Unitary conjugation is an isometry in $p$-Schatten class, for every $p$. Indeed,  if
$V$ is unitary and $A$ compact, then $(VAV^*)^*(VAV^*) = VA^*AV^*$ is unitarily
equivalent to $A^*A$, hence has the same eigenvalues; so $VAV^*$ and $A$ have the
same singular values. Therefore, $\|VAV^*\|_{\mathcal{T}^p} = \|A\|_{\mathcal{T}^p}$.

\item[(F2)]  If $A$ has
rank $\le n$ with singular values $s_1,\dots,s_n$, and $1\le p\le 2$, then by
H\"older's inequality applied to $\sum_i s_i^p\cdot 1$ (exponent $\frac{2}{p}$ and its conjugate),  we have
\begin{equation*}
\|A\|_{\mathcal{T}^p} = \Big(\sum_{i=1}^n s_i^p\Big)^{1/p}
\le n^{\frac1p-\frac12}\Big(\sum_{i=1}^n s_i^2\Big)^{1/2} = n^{\frac1p-\frac12}\|A\|_{\mathcal{T}^2}.
\end{equation*}

\item[(F3)] The Hilbert-Schmidt norm of the rank-one operator  $\varphi\otimes\psi$  is 
$\|\varphi\otimes\psi\|_{\mathcal{T}^2} = \|\varphi\|_{\mathcal{H}}\|\psi\|_{\mathcal{H}}$. Indeed, let $(e_n)$ be an
orthonormal basis of $\mathcal{H}$. We have 
\begin{align*}
\|\varphi\otimes\psi\|_{\mathcal{T}^2}^2 &=\sum\limits_n
\|(\varphi\otimes\psi)e_n\|_{\mathcal{H}}^2 \\
&=\sum\limits_n
\|\langle e_n,\psi \rangle\varphi\|_{\mathcal{H}}^2 \\
&=\|\varphi\|_{\mathcal{H}}^2\sum\limits_n
|\langle e_n,\psi \rangle|^2 \\
&=  \|\psi\|_{\mathcal{H}}^2\|\varphi\|_{\mathcal{H}}^2\\ 
& \text{ (by the Parseval equality). } 
\end{align*}

\item[(F4)] Finite-rank operators are norm-dense in $\mathcal{T}^p(\mathcal{H})$ for $1\le
p<\infty$.  Indeed, for $A=\sum\limits_n s_n\varphi_n\otimes
\psi_n$ with singular values $(s_n)$, the finite-rank truncations $A_N=\sum\limits_{n\le N}
s_n\varphi_n\otimes\psi_n$ satisfy 
$$\|A-A_N\|_{\mathcal{T}^p}^p=\sum\limits_{n>N}s_n^p$$ which tends to $0$ at infinity since the series $\sum\limits_{n}s_n^p$ converges.

\item[(F5)] For $x,y\in\Xi$, we have $\alpha_x\circ \alpha_y=\alpha_{x+y}$. Indeed, 
$
\alpha_x(\alpha_y(A)) = U_x(U_yAU_y^*)U_x^* = (U_xU_y)A(U_xU_y)^*
= m(x,y)U_{x+y}A\,\overline{m(x,y)}\,U_{x+y}^* = |m(x,y)|^2\alpha_{x+y}(A)=\alpha_{x+y}(A),$
since $|m(x,y)|=1$. 
\end{itemize}

\bigskip

We prove the following lemma.
\begin{lemma}
\label{lem:shiftcont}
For every $A\in\mathcal{T}^p(\mathcal{H}),\, 1\le p\le 2$, the map $x\mapsto\alpha_x(A)$  defined from  $\Xi$ into $\mathcal{T}^p(\mathcal{H})$
is continuous  and $\|\alpha_x(A)\|_{\mathcal{T}^p} = \|A\|_{\mathcal{T}^p}$ for all $x\in \Xi$.
\end{lemma}

\begin{proof}
\begin{itemize}
\item Applying (F1) with $V=U_x$ yields the isometry assertion. 

\item We now show the continuity at $x=0$ for rank-one operators. Let $A=\varphi\otimes\psi$. We have 
\begin{align*}
 U_xAU_x^*\phi &= U_x\langle U_x^*\phi,\psi\rangle\varphi \\
 &=
\langle\phi,U_x\psi\rangle U_x\varphi,
\end{align*}
that is
\begin{equation*}
\alpha_x(\varphi\otimes\psi) = (U_x\varphi)\otimes(U_x\psi).
\end{equation*}
Then, 
\begin{equation*}
\alpha_x(A)-A = (U_x\varphi-\varphi)\otimes(U_x\psi) + \varphi\otimes(U_x\psi-\psi).
\end{equation*}
 We have 
\begin{align*}
\|\alpha_x(A)-A\|_{\mathcal{T}^2} &\le \|U_x\varphi-\varphi\|_{\mathcal{H}}\|U_x\psi\|_{\mathcal{H}} +
\|\varphi\|_{\mathcal{H}}\|U_x\psi-\psi\|_{\mathcal{H}}\\
& \text{(by the triangle inequality and  (F3))}\\
&= \|U_x\varphi-\varphi\|_{\mathcal{H}}\|\psi\|_{\mathcal{H}} +
\|\varphi\|_{\mathcal{H}}\|U_x\psi-\psi\|_{\mathcal{H}}\\
& (\text{ by the unitarity of } U_x).
\end{align*}
The right hand side tends to 0 when $x$  tends to 0 in $\Xi$ 
by the strong continuity of the projective representation $U$. Using  (F2) with $n=2$, we obtain
$$\|\alpha_x(A)-A\|_{\mathcal{T}^p} \le 2^{\frac1p-\frac12}\|\alpha_x(A)-A\|_{\mathcal{T}^2}
\quad \text{ which tends to } 0$$ as $x \text{ tends to } 0$. 

\item  For a finite-rank operator $A=\sum\limits_{i=1}^n\varphi_i\otimes\psi_i$, the linearity of
$\alpha_x$ and the triangle inequality give
$$
\|\alpha_x(A)-A\|_{\mathcal{T}^p} \le \sum_{i=1}^n
\|\alpha_x(\varphi_i\otimes\psi_i)-\varphi_i\otimes\psi_i\|_{\mathcal{T}^p}
\quad \text{ which tends to } 0$$ 
as $x \text{  tends to } 0
$ as a finite sum of terms each vanishing by the previous paragraph.

\item Consider a general 
 $A\in\mathcal{T}^p(\mathcal{H})$.   Fix $\varepsilon>0$. By (F4),  choose a
finite-rank operator $A'$ such that $\|A-A'\|_{\mathcal{T}^p}<\frac{\varepsilon}{3}$.  By the previous
paragraph choose a neighbourhood $V\ni 0$ such that  $\|\alpha_x(A')-A'\|_{\mathcal
T^p}<\frac{\varepsilon}{3}$ for $x\in V$. For $x\in V$, we  have 

\begin{align*}
\|\alpha_x(A)-A\|_{\mathcal{T}^p}& \le \|\alpha_x(A)-\alpha_x(A')\|_{\mathcal
T^p} + \|\alpha_x(A')-A'\|_{\mathcal
T^p} + \|A'-A\|_{\mathcal{T}^p} \\
&\le \|A-A'\|_{\mathcal
T^p} + \|\alpha_x(A')-A'\|_{\mathcal
T^p} + \|A'-A\|_{\mathcal{T}^p}\\
&(\text{ by the linearity of $\alpha_x$ and its  isometry
property})\\
&< \frac{\varepsilon}{3}+\frac{\varepsilon}{3}+\frac{\varepsilon}{3}=\varepsilon.
\end{align*}
Hence, the map $x\mapsto \alpha_x$ is continuous at $0$.
\item We can now prove the continuity at an arbitrary $x_0$. By (F5) and the isometry property, we have 
\begin{align*}
\|\alpha_x(A)-\alpha_{x_0}(A)\|_{\mathcal{T}^p}
&= \big\|\alpha_{x_0}\big(\alpha_{x-x_0}(A)-A\big)\big\|_{\mathcal{T}^p}\\
&(\text{ by (F5)})\\
&= \|\alpha_{x-x_0}(A)-A\|_{\mathcal{T}^p}\\
&(\text{by isometry})
 \end{align*}
 which tends to $0$ as $x$ goes to $x_0$
by continuity at $0$. 
\end{itemize}
\end{proof}

\begin{lemma}
\label{lem:precompacteq}
If $F\subset\mathcal{T}^p(\mathcal{H}),\, 1\le p\le 2,$ is precompact, then $F$ is uniformly $\mathcal
T^p$-equicontinuous.
\end{lemma}

\begin{proof}

Fix $\varepsilon>0$.
  By precompactness hence total boundedness  of $F$,   with
$\eta=\frac{\varepsilon}{3}$, there exist $A_1,\dots,A_n\in\mathcal{T}^p(\mathcal{H})$ such that
\begin{equation}\label{eq:varepsilonnet}
F \subset \bigcup_{i=1}^n\left\{T:\|T-A_i\|_{\mathcal{T}^p}<\frac{\varepsilon}{3}\right\}.
\end{equation}
 
 By Lemma
\ref{lem:shiftcont}, applied to each of the  operators
$A_1,\dots,A_n$, we obtain that  for each $i$,  there exists a neighbourhood $V_i\ni0$ such that
$$\|\alpha_x(A_i)-A_i\|_{\mathcal{T}^p}<\frac{\varepsilon}{3}$$ for $x\in V_i$. Set
$V=\bigcap\limits_{i=1}^n V_i$.  Being a finite intersection of neighbourhoods of
$0$, $V$ is again a neighbourhood of $0$.
 
  Let $A\in F$ and $x\in V$. By
(\ref{eq:varepsilonnet}) pick $i$ such that $\|A-A_i\|_{\mathcal{T}^p}<\frac{\varepsilon}{3}$. Then
\begin{align*}
\|\alpha_x(A)-A\|_{\mathcal{T}^p} &\le
\|\alpha_x(A)-\alpha_x(A_i)\|_{\mathcal{T}^p} +
\|\alpha_x(A_i)-A_i\|_{\mathcal{T}^p} +
\|A_i-A\|_{\mathcal{T}^p}\\
&=\|A_i-A\|_{\mathcal{T}^p}+
\|\alpha_x(A_i)-A_i\|_{\mathcal{T}^p} +
\|A_i-A\|_{\mathcal{T}^p}\\
&< \frac{\varepsilon}{3}+\frac{\varepsilon}{3}+\frac{\varepsilon}{3}=\varepsilon.
\end{align*}
Thus, $F$ is uniformly $\mathcal
T^p$-equicontinuous.
\end{proof}

\begin{lemma}
\label{lem:transfer}
 Let $1\le p\le 2$ and let $q$ be the conjugate exponent of $p$. If
$F\subset\mathcal{T}^p(\mathcal{H})$ is bounded and uniformly $\mathcal{T}^p$-equicontinuous, then
$ \mathcal{F}_U(F) \subset L^q(\widehat{\Xi})$ has uniform $L^q(\widehat{\Xi})$-decay.
\end{lemma}

\begin{proof}
Fix $\varepsilon>0$. By uniform $\mathcal{T}^p$-equicontinuity of $F$, there is an open
neighbourhood $V\ni 0$ in $\Xi$ such that
\begin{equation} \label{eq:alpha}
\|A - \alpha_y(A)\|_{\mathcal{T}^p} < \frac{\varepsilon^{1/q}}{2} \qquad \text{for all } A\in F,\ y\in
V. 
\end{equation}
Consider a function  $\phi\in L^1(\Xi)$ such that  $\phi\ge0$, $\displaystyle\int_{\Xi}\phi=1$, with $\supp\phi\subset V$
(the existence of $\phi$ is a consequence of Urysohn's lemma). For
$A\in F$, we have
\begin{equation*}
A - \phi*A = \int_{\Xi} \phi(y)\big(A-\alpha_y(A)\big)dy. 
\end{equation*}
 Therefore, 
\begin{equation}
\|A-\phi*A\|_{\mathcal{T}^p} \le \int_{\Xi}\phi(y)\,\|A-\alpha_y(A)\|_{\mathcal
T^p}dy < \frac{\varepsilon^{1/q}}{2}\int_{\Xi}\phi(y)dy=\frac{\varepsilon^{1/q}}{2}. 
\end{equation}
By \textbf{(Conv)} and \textbf{(HY)}, we have
\begin{equation}\label{eq:ConvHY}
\big\|\mathcal{F}_U(A) - \mathcal{F}_\sigma(\phi)\mathcal{F}_U(A)\big\|_{L^q(\widehat{\Xi})}
= \big\|\mathcal{F}_U(A-\phi*A)\big\|_{L^q(\widehat{\Xi})}
\le \|A-\phi*A\|_{\mathcal{T}^p} < \frac{\varepsilon^{1/q}}{2}. 
\end{equation}
By the classical Riemann-Lebesgue lemma for the symplectic Fourier transform,
$\mathcal{F}_\sigma(\phi)\in C_0(\widehat{\Xi})$. Hence, there is a compact set
$K\subset\widehat{\Xi}$ such that $|\mathcal{F}_\sigma(\phi)(\xi)|\le\frac{1}{2}$ for $\xi\notin K$. For such
$\xi$, $|1-\mathcal{F}_\sigma(\phi)(\xi)|\ge\frac{1}{2}$.
Therefore,  for $\xi\notin K$,
\begin{align*}
|\mathcal{F}_U(A)(\xi)|\le 2|\mathcal{F}_U(A)(\xi)||1-\mathcal{F}_\sigma(\phi)(\xi)|. 
\end{align*}

We have 
\begin{align*}
\left(\int_{\widehat{\Xi}\setminus K}|\mathcal{F}_U(A)(\xi)|^qd\xi\right)^{1/q}& \le 2\left(\int_{\widehat{\Xi}\setminus K}\big|\mathcal{F}_U(A)(\xi)\,(1-\mathcal{F}_\sigma(\phi)(\xi))\big|^q\,d\xi\right)^{1/q}\\
&\le 2\left(\int_{\widehat{\Xi}}\big|\mathcal{F}_U(A)(\xi)\,(1-\mathcal{F}_\sigma(\phi)(\xi))\big|^q\,d\xi\right)^{1/q}\\
&< 2\frac{\varepsilon^{1/q}}{2}=\varepsilon^{1/q}\qquad(\text{by (\ref{eq:ConvHY})}).
\end{align*}
Thus, $\displaystyle\int_{\widehat{\Xi}\setminus K}|\mathcal{F}_U(A)(\xi)|^q\,d\xi<\varepsilon$; that is, $\mathcal{F}_U(F)$ has uniform $L^q(\widehat{\Xi})$-decay.
\end{proof}

\begin{lemma}
\label{lem:converse}
 Let $F\subset\mathcal{T}^2(\mathcal{H})$ be bounded. If $\mathcal{F}_U(F)\subset
L^2(\widehat{\Xi})$ has uniform $L^2(\widehat{\Xi})$-decay, then $F$ is uniformly $\mathcal{T}^2$-equicontinuous.
\end{lemma}

\begin{proof}
Set $M=\sup\limits_{A\in F}\|A\|_{\mathcal{T}^2}<\infty$.  By \textbf{(PL)},
$\sup\limits_{A\in F}\|\mathcal{F}_U(A)\|_{L^2}=M$ as well. If $M=0$ then $F=\{0\}$ and there
is nothing to prove, so assume $M>0$.

Fix $\varepsilon>0$. By uniform $L^2(\widehat{\Xi})$-decay of $\mathcal{F}_U(F)$, there exists a compact
$K\subset\widehat{\Xi}$ such that
\begin{equation}\label{eq:decayhyp}
\int_{\widehat{\Xi}\setminus K}|\mathcal{F}_U(A)(\xi)|^2\,d\xi < \frac{\varepsilon^2}{8}\quad\text{for all } A\in F.
\end{equation}

For $A\in F$ and $x\in\Xi$, we have
\begin{equation*}
\|\alpha_x(A)-A\|_{\mathcal{T}^2}^2 = \big\|\mathcal{F}_U(\alpha_x(A))-\mathcal{F}_U(A)\big\|_{L^2}^2.
\end{equation*}
Since,  $\mathcal{F}_U(\alpha_x(A))(\xi)=\sigma(x,\xi)\mathcal{F}_U(A)(\xi)$, we have 
\begin{align*}
\|\alpha_x(A)-A\|_{\mathcal{T}^2}^2 &= \int_{\widehat{\Xi}}|\sigma(x,\xi)-1|^2\,|\mathcal{F}_U(A)(\xi)|^2d\xi\\
&= \int_K|\sigma(x,\xi)-1|^2\,|\mathcal{F}_U(A)(\xi)|^2\xi + \int_{\widehat{\Xi}\setminus K}|\sigma(x,\xi)-1|^2\,|\mathcal{F}_U(A)(\xi)|^2d\xi.
\end{align*}

 Since $|\sigma(x,\xi)|=1$, we have $|\sigma(x,\xi)-1|\le2$. So,  by \eqref{eq:decayhyp}, 
\begin{equation}\label{k1}
\int_{\widehat{\Xi}\setminus K}|\sigma(x,\xi)-1|^2\,|\mathcal{F}_U(A)(\xi)|^2\,d\xi \le 4\int_{\widehat{\Xi}\setminus K}|\mathcal{F}_U(A)(\xi)|^2\,d\xi
< 4\cdot\frac{\varepsilon^2}{8} = \frac{\varepsilon^2}{2}.
\end{equation}

Furthermore, the pairing $\sigma:\Xi\times\widehat{\Xi}\to\mathbb T$ is jointly
continuous, and $\sigma(0,\xi)=1$ for every $\xi\in\widehat{\Xi}$ (because $x\mapsto
\sigma(x,\cdot)$ is a continuous group homomorphism $\Xi\to\widehat{\Xi}$, so $\sigma(0,\cdot)$
is the trivial character in $\widehat{\Xi}$). Since $K$ is compact, joint continuity gives uniform continuity. So, there is an open
neighbourhood $V\ni 0$ such that
\begin{equation*}
\sup_{\xi\in K}|\sigma(x,\xi)-1| < \frac{\varepsilon}{2M}\quad\text{for all } x\in V.
\end{equation*}
Using the fact that 
$$\int_K|\mathcal{F}_U(A)(\xi)|^2\,d\xi\le\|\mathcal{F}_U(A)\|_{L^2}^2\le M^2,$$  we obtain, for every
$A\in F$ and every $x\in V$,
\begin{align*}
\int_K|\sigma(x,\xi)-1|^2\,|\mathcal{F}_U(A)(\xi)|^2\,d\xi & \le \Big(\frac{\varepsilon}{2M}\Big)^2
\int_K|\mathcal{F}_U(A)(\xi)|^2\,d\xi \\
&\le \Big(\frac{\varepsilon}{2M}\Big)^2M^2
 = \frac{\varepsilon^2}{4}
< \frac{\varepsilon^2}{2}.
\end{align*}

Finally,  for every $A\in F$ and every
$x\in V$, we have
\begin{equation*}
\|\alpha_x(A)-A\|_{\mathcal{T}^2}^2 < \frac{\varepsilon^2}{2}+\frac{\varepsilon^2}{2}
= \varepsilon^2,
\end{equation*}
i.e.\ $\|\alpha_x(A)-A\|_{\mathcal{T}^2}<\varepsilon$. Thus, $F$ is uniformly $\mathcal{T}^2$-equicontinuous.
\end{proof}

\begin{corollary}
Let $F\subset\mathcal{T}^2(\mathcal{H})$ be bounded. Then $F$ is
uniformly $\mathcal{T}^2$-equicontinuous if and only if $\mathcal{F}_U(F)$ has uniform
$L^2(\widehat{\Xi})$-decay.
\end{corollary}

\begin{proof}
Combine Lemma \ref{lem:transfer} (with $p=q=2$) and  Lemma \ref{lem:converse}.
\end{proof}

We can now state and prove the main result which charaterizes precompactness in the Hilbert-Schmidt class. 

\begin{theorem}
\label{thm:main}
Let $F\subset\mathcal{T}^2(\mathcal{H})$ be bounded. The following
assertions are equivalent.
\begin{enumerate}
\item[(i)] $F$ is precompact.
\item[(ii)] $F$ is uniformly $\mathcal{T}^2$-equicontinuous \emph{and} $\mathcal{F}_U(F)$ is
uniformly $L^2(\widehat{\Xi})$-equicontinuous.
\item[(iii)] $\mathcal{F}_U(F)$ has uniform $L^2(\widehat{\Xi})$-decay \emph{and} $\mathcal{F}_U(F)$ is
uniformly $L^2(\widehat{\Xi})$-equicontinuous.
\end{enumerate}
\end{theorem}

\begin{proof}
 $\mathcal{F}_U:\mathcal{T}^2(\mathcal{H})\to L^2(\widehat{\Xi})$ is a unitary isomorphism. Since a unitary map carries precompact sets to precompact
sets and vice versa, we have the equivalence
\begin{equation}\label{eq:equivalence}
F \text{ precompact in } \mathcal{T}^2(\mathcal{H})
\quad\Longleftrightarrow\quad
\mathcal{F}_U(F) \text{ precompact in } L^2(\widehat{\Xi}). 
\end{equation}
 
\text{(i)$\Rightarrow$(ii)}. Assume $F$ is precompact. By Lemma
\ref{lem:precompacteq}, $F$ is uniformly $\mathcal{T}^2$-equicontinuous. Also
$\mathcal{F}_U(F)$ is precompact in $L^2(\widehat{\Xi})$ by \eqref{eq:equivalence}. The  precompactness of $\mathcal{F}_U(F)$
forces $\mathcal{F}_U(F)$ to be uniformly $L^2(\widehat{\Xi})$-equicontinuous (Theorem \ref{thm:gorka-pego} applied to the LCA  group $\widehat{\Xi}$).
 
\bigskip

\text{(ii)$\Rightarrow$(iii)}. Assume (ii). By Lemma
\ref{lem:transfer} (with $p=q=2$), uniform $\mathcal{T}^2$-equicontinuity of $F$
implies that $\mathcal{F}_U(F)$ has uniform $L^2(\widehat{\Xi})$-decay. The second clause of (iii)
is the same as the second clause of (ii).
 
\text{(iii)$\Rightarrow$(i)}. Assume $\mathcal{F}_U(F)$ has uniform $L^2(\widehat{\Xi})$-decay and
is uniformly $L^2(\widehat{\Xi})$-equicontinuous.

Since $F$ is bounded, then by \textbf{(PL)},  $\mathcal{F}_U(F)$ is bounded.  
 Moreover, by hypothesis it has uniform
$L^2(\widehat{\Xi})$-decay and is uniformly $L^2(\widehat{\Xi})$-equicontinuous. By the Weil compactness criterion (Theorem\ref{thm:weil-compactness}),  $\mathcal{F}_U(F)$ is precompact in
$L^2(\widehat{\Xi})$. Therefore,  $\mathcal{F}_U(F)$ is totally bounded. 
 Fix $\varepsilon>0$. There exist  $g_1,\dots,g_n\in L^2(\widehat{\Xi})$ such that
$$\mathcal{F}_U(F) \subset \bigcup_{i=1}^n\{h:\|h-g_i\|_{L^2}<\varepsilon\}.$$
 Set
$A_i=\mathcal{F}_U^{-1}(g_i)\in\mathcal{T}^2(\mathcal{H})$. For $A\in F$, pick $i$ such that $\|\mathcal{F}_U(A)-g_i\|_{L^2}
<\varepsilon$. Then, 
\begin{align*}
\|A-A_i\|_{\mathcal{T}^2} &= \|\mathcal{F}_U(A-A_i)\|_{L^2}\\
& = \|\mathcal{F}_U(A)-\mathcal{F}_U(A_i)\|_{L^2} \\
&=
\|\mathcal{F}_U(A)-g_i\|_{L^2} < \varepsilon.
\end{align*}
Thus  $F\subset \union\limits_{i=1}^n\{ A: \|A-A_i\|<\varepsilon\}$. 
That is, $F$ is totally bounded, hence  precompact.
\end{proof}

\section{Applications}
\label{sec:applications}
 
Throughout this section $\Xi$ is a genuine LCA group, so the full translation/modulation duality  is available.  
 
\subsection{Quantum information: precompactness of thermal state families}
 
Take $\Xi=\mathbb R^{2n}$, $U$ the Weyl displacement-operator representation on
$\mathcal{H}=L^2(\mathbb{R}^n)$, so that $\mathcal{F}_U(A)(z)=\text{tr}(AU_z^*)=:\chi_A(z)$ is the quantum characteristic function of $A$ used throughout
continuous-variable quantum information \cite{Weedbrook2012}. For the $n$-mode
isotropic thermal state $\rho_{\bar{n}}$ with mean photon number $\bar n\ge0$, we know that
\begin{equation*}
\chi_{\rho_{\bar{n}}}(z) = e^{-(\bar n+\frac12)\|z\|^2}, \quad z\in\mathbb{R}^{2n}.
\end{equation*}

 Let $\bar n_{\min}$ and $\bar n_{\max}$ denote respectively the minimum and the maximum values of $\bar{n}$. Assume $\bar n_{\max}<\infty$ (an infinite $\bar{n}$ would require infinite energy, which is physically impossible because any real power source is strictly finite and excessive energy density  destroys the physical system \cite{Roch2025}). Consider the family $F=\{\rho_{\bar n}:\bar n\in
[\bar n_{\min},\bar n_{\max}]\}$. 
\begin{proposition}
\label{prop:thermal}
 The family $F$ is precompact in $\big(\mathcal{T}^2(L^2(\mathbb
R^n)),\|\cdot\|_{\mathcal{T}^2}\big)$.
\end{proposition}
 
\begin{proof}
Set $a=2\bar n+1\in [a_{\min}, a_{\max}]$ with $a_{\min}=2\bar n_{\min}+1$ and $a_{\max}=2\bar n_{\max}+1$.
 \begin{itemize} 
\item \emph{Boundedness.} 
We have 
\begin{align*}
\|\chi_{\rho_{\bar n}}\|_{L^2(\mathbb R^{2n})}^2&=\int_{R^{2n}}
e^{-a\|z\|^2}dz\\
&=\left(\frac{\pi}{a}\right)^n\\
&\le \left(\frac{\pi}{a_{\min}}\right)^n.
\end{align*}
Therefore, by the Plancherel theorem \textbf{(PL)},
$$\|\rho_{\bar n}\|_{\mathcal{T}^2}\le \left(\frac{\pi}{a_{\min}}\right)^{\frac{n}{2}}$$ for all  $\bar
n\in[\bar n_{\min},\bar n_{\max}]$.
 
 \item \emph{Uniform $\mathcal{T}^2$-equicontinuity of ${F}$.} For $\eta\in\mathbb{R}^{2n}$, we have
\begin{align*}
\|\alpha_\eta(\rho_{\bar n})-\rho_{\bar n}\|_{\mathcal{T}^2}^2 &=
\|\mathcal{F}_U(\alpha_\eta(\rho_{\bar n}))-\chi_{\rho_{\bar n}}\|_{L^2}^2 \\
&(\text{by the Plancherel theorem})\\
&=\|\sigma(\eta, \xi)\mathcal{F}_U(\rho_{\bar n})-\chi_{\rho_{\bar n}}\|_{L^2}^2 \\
&=\int_{\mathbb R^{2n}}|\sigma(\eta,\xi)-1|^2\,e^{-a\|\xi\|^2}d\xi.
\end{align*}
Set
$$ h(\eta,\bar n)=\int_{\mathbb R^{2n}}|\sigma(\eta,\xi)-1|^2\,e^{-a\|\xi\|^2}\,d\xi.$$ 
The integrand is continuous because its  factors are individually continuous. 
Since $|\sigma(\eta,\xi)|=1$, $|\sigma(\eta,\xi)-1|\le 2$. So,
the integrand is dominated by
$4e^{-a_{\min}\|\xi\|^2}$. The latter is  integrable (as a Gaussian) and independent of $(\eta,\bar{n})$.  By the Lebesgue dominated convergence theorem, $h$ is jointly continuous on
$\{\eta: \|\eta\|\le1\}\times[\bar n_{\min},\bar n_{\max}]$, with $h(0,\bar n)=0$ for every
$\bar n$. Joint continuity on  compact set implies uniform
continuity. Therefore,   $\sup\limits_{\bar n}h(\eta,\bar n)$ tends to $0$ as $\eta$ goes $0$.  Thus, $F$ is  uniformly
$\mathcal{T}^2$-equicontinuous.
 
\item \emph{Uniform $L^2(\widehat{\Xi})$-equicontinuity of $\mathcal{F}_U(F)$.} Similarly,
\begin{equation*}
\|\chi_{\rho_{\bar n}}(\cdot+\eta)-\chi_{\rho_{\bar n}}\|_{L^2}^2 =
\int_{\mathbb R^{2n}}\big|e^{-\frac a2\|\xi+\eta\|^2}-e^{-\frac a2\|\xi\|^2}
\big|^2\,d\xi. 
\end{equation*}
Set $$k(\eta,\bar n)=\int_{\mathbb R^{2n}}\big|e^{-\frac a2\|\xi+\eta\|^2}-e^{-\frac a2\|\xi\|^2}
\big|^2\,d\xi.$$
We use the classical inequality $$|u - v|^2 \le 2|u|^2 + 2|v|^2$$ togheter with the constraint $\|\eta\| \le 1$ and the fact $a \ge a_{\min} > 0$, to obtain that the integrand can be bounded uniformly as follows:
\begin{equation*}
\big|e^{-\frac a2\|\xi+\eta\|^2}-e^{-\frac a2\|\xi\|^2}\big|^2 \le 2e^{-a_{\min}\|\xi+\eta\|^2} + 2e^{-a_{\min}\|\xi\|^2}
\end{equation*}
Applying the triangle inequality, we obtain  $\|\xi+\eta\| \ge \|\xi\| - 1$ and this  conducts to
\begin{equation*}
\big|e^{-\frac a2\|\xi+\eta\|^2}-e^{-\frac a2\|\xi\|^2}\big|^2 \le  2e^{a_{\min}} e^{-\frac{1}{2}a_{\min}\|\xi\|^2} + 2e^{-a_{\min}\|\xi\|^2}. 
\end{equation*}
Since the right hand side is independent of $(\eta,\bar n)$ and integrable,  
 the Lebesgue Dominated Convergence Theorem applied. Moreover,  
\begin{equation*}
k(0, \bar{n}) = \int_{\mathbb R^{2n}}\big|e^{-\frac a2\|\xi\|^2}-e^{-\frac a2\|\xi\|^2}\big|^2\,d\xi = 0.
\end{equation*}
Thus, similarly to the above paragraph, $\sup\limits_{\bar n}k(\eta,\bar n)$ tends to $0$ as $\eta$ goes $0$.
Hence, $\mathcal{F}_U(F)$  is uniformly $L^2(\widehat{\Xi})$-equicontinuous.

By Theorem \ref{thm:main}, $F$ is
precompact.
\end{itemize}
\end{proof}

\subsection{Quantum statistics: tightness of tomographic state estimators}
 
In quantum tomography, an unknown state $\rho$ is estimated from $n$ i.i.d.\
measurement outcomes $Z_1,\dots,Z_n\in \widehat{\Xi}$ whose distribution yields an
unbiased pointwise estimator of the characteristic function
$$\widehat\chi_n(\xi)=\displaystyle\frac{1}{n}\sum_{j=1}^n\overline{\sigma(Z_j,\xi)}.$$ 
This is the direct
quantum analogue of the empirical characteristic function used in classical
deconvolution density estimation \cite{ArtilesGillGuta2005}. One estimates $\rho$ by taking the 
Fourier inverse of  a truncated version of $\widehat\chi_n$:
\begin{equation*}
\widehat\rho_n = \mathcal{F}_U^{-1}\big(\widehat\chi_n\cdot1_{B_T}\big) =
\int_{\|\xi\|\le T}\widehat\chi_n(\xi)\,U_\xi\,d\xi, 
\end{equation*}
where $B_T=\{\xi \in \widehat{\Xi}:\|\xi\|\le T\}$ and  $1_{B_T}$ is the characteristic function of the set $B_T$.

\begin{proposition}
\label{prop:tightness}
Fix $T^*<\infty$ and suppose $\{\widehat\chi_n\cdot1_{B_{T^*}}\}_n$ is (almost surely or in
probability) uniformly $L^2$-equicontinuous on the fixed compact set $B_{T^*}$. Then
$\{\widehat\rho_n\}_n$, with $\widehat\rho_n=\mathcal{F}_U^{-1}(\mathbf1_{B_{T^*}}
\widehat\chi_n)$, is  precompact in
$\mathcal{T}^2(\mathcal{H})$.
\end{proposition}
 
\begin{proof}
\begin{itemize}
\item \emph{Boundedness.}
Since $|\sigma(Z_j,\xi)|=1$,  $|\widehat\chi_n(\xi)|\le 1$. Therefore,  
  $$\|1_{B_{T^*}}\widehat\chi_n\|_{L^2}\le
|B_{T^*}|^{1/2},$$ 
where $|B_{T^*}|$ is the volume of $B_{T^*}$.
By the Plancherel theorem (PL), 
$$\|\widehat\rho_n\|_{\mathcal{T}^2}\le|B_{T^*}|^{1/2}\quad \text{ for all } n.$$ 
 
\item \emph{Uniform $L^2(\widehat{\Xi})$-decay.} $\mathcal{F}_U(\widehat\rho_n)=1_{B_{T^*}}\widehat\chi_n$
vanishes identically outside the fixed compact $B_{T^*}$, so uniform
$L^2$-decay of $\{\mathcal{F}_U(\widehat\rho_n)\}_n$ holds trivially with the same compact set.
 
\item \emph{Uniform $L^2(\widehat{\Xi})$-equicontinuity.} On $B_{T^*}$ this is the hypothesis; off $B_{T^*}$,
$\mathcal{F}_U(\widehat\rho_n)\equiv0$, so uniform $L^2(\widehat{\Xi})$-equicontinuity
of $\{\mathcal{F}_U(\widehat\rho_n)\}_n$ holds on all  $\widehat{\Xi}$.
 \end{itemize}
By Theorem \ref{thm:main}, $\{\widehat\rho_n\}$ is precompact.
\end{proof}

\section{Conclusion}

In this paper we established an operator-theoretic analogue of Pego's compactness
theorem within Fulsche-Galke's framework of quantum harmonic analysis on general
locally compact abelian phase spaces. 
We illustrated the applicability of the abstract criterion in two settings drawn
from quantum information and quantum statistics.  

\noindent Our criterion characterizes precompactness almost entirely through the transform side.  It would be of interest to
find an intrinsic notion of decay at infinity directly for an operator, not mediated by the  Fourier-Weyl transform, that could serve as an operator-side analogue of the classical decay condition. For this purpose the Halvdansson's extension of quantum harmonic analysis beyond the
locally compact abelian setting \cite{Halvdansson2025} can serve as a guide.

\end{document}